\documentclass[11pt]{amsart}
\usepackage{mathrsfs,float,amssymb,tikz-network,hyperref}
\hypersetup{
    colorlinks,
    citecolor=blue,
    filecolor=blue,
    linkcolor=blue,
    urlcolor=blue
}
\usepackage[noabbrev,capitalize,nameinlink]{cleveref}
\newtheorem{theorem}{Theorem}[section]
\newtheorem{lemma}[theorem]{Lemma}
\newtheorem{corollary}[theorem]{Corollary}

\theoremstyle{definition}
\newtheorem{definition}[theorem]{Definition}
\newtheorem{proposition}[theorem]{Proposition}
\newtheorem{conjecture}[theorem]{Conjecture}
\newtheorem{oproblem}[theorem]{Problem}

\usepackage[margin=1in]{geometry}

\begin{document}

\title[Rado Games for Equations with Radicals]{Maker-Breaker Rado games for equations with radicals}
\author[Gaiser]{Collier Gaiser}
\address[Gaiser]{Department of Mathematics, University of Denver, Denver, CO 80208}
\email{\href{mailto:collier.gaiser@du.edu}{collier.gaiser@du.edu}}
\thanks{}
\author[Horn]{Paul Horn}
\address[Horn]{Department of Mathematics, University of Denver, Denver, CO 80208}
\email{\href{mailto:paul.horn@du.edu}{paul.horn@du.edu}}





\begin{abstract}
We study two-player positional games where Maker and Breaker take turns to select a previously unoccupied number in $\{1,2,\ldots,n\}$. Maker wins if the numbers selected by Maker contain a solution to the equation
\[
x_1^{1/\ell}+\cdots+x_k^{1/\ell}=y^{1/\ell}
\]
where $k$ and $\ell$ are integers with $k\geq2$ and $\ell\neq0$, and Breaker wins if they can stop Maker. Let $f(k,\ell)$ be the smallest positive integer $n$ such that Maker has a winning strategy when $x_1,\ldots,x_k$ are not necessarily distinct, and let $f^*(k,\ell)$ be the smallest positive integer $n$ such that Maker has a winning strategy when $x_1,\ldots,x_k$ are distinct. 

When $\ell\geq1$, we prove that, for all $k\geq2$, $f(k,\ell)=(k+2)^\ell$ and $f^*(k,\ell)=(k^2+3)^\ell$; when $\ell\leq-1$, we prove that $f(k,\ell)=[k+\Theta_k(1)]^{-\ell}$ and $f^*(k,\ell)=[\exp(O_k(k\log k))]^{-\ell}$. Our proofs use elementary combinatorial arguments as well as results from number theory and arithmetic Ramsey theory.
\end{abstract}
\maketitle
\section{Introduction}\label{Section:Introduction}
Let $\mathcal{F}$ be a family of finite subsets of $\mathbb{N}:=\{1,2,\ldots\}$ and $n\in\mathbb{N}$. Maker-Breaker games played on $[n]:=\{1,2,\ldots,n\}$ with winning sets $\mathcal{F}$ are two-player positional games where Maker and Breaker take turns to select a previously unoccupied number in $[n]$. Maker goes first. Maker wins if they can occupy a set in $\mathcal{F}$ and Breaker wins otherwise. The van der Waerden games introduced by Beck \cite{Beck1981} are games of this type. In van der Waerden games, $\mathcal{F}$ is the set of $k$-term arithmetic progressions for a fixed $k$. These games were motivated by a result of van der Waerden's theorem \cite{vanderWaerden1927} which says that if $\mathbb{N}$ is partitioned into two classes, then one of them contains arbitrarily long arithmetic progressions. By the compactness principle \cite[Chapter 1]{GrahamButler2015} (see also \cite[Section 2.1]{LandmanRobertson2014}) and strategy stealing \cite[Section 5]{Beck2008} (see also \cite[Chapter 1]{HKSS2014}), Maker can win the van der Waerden games if $n$ is large enough. Therefore, one would naturally want to find the smallest $n$ such that Maker can win the van der Waerden games. Beck \cite{Beck1981} proved that, for any given $k$, the smallest $n$ such that Maker has a winning strategy for the van der Waerden games is between $2^{k-7k^{7/8}}$ and $k^32^{k-4}$.

Recently, Kusch, Ru\'{e}, Spiegel, and Szab\'{o} \cite{KRSS2019} studied a generalization of van der Waerden games called Rado games. In Rado games, $\mathcal{F}$ is the set of solutions to a system of linear equations. By Rado's theorem \cite{Rado1933}, if $n$ is large enough, then Maker is guaranteed to win the Rado games if the system of linear equations satisfies the so-called column condition \cite[Chapter 10]{GrahamButler2015}. Kusch, Ru\'{e}, Spiegel, and Szab\'{o} allowed maker to select $q\geq1$ numbers each round and derived asymptotic thresholds of $q$ for Breaker's to win. Their result on 3-term arithmetic progressions was later improved by Cao et al. \cite{CCELSXY2022}. Hancock \cite{Hancock2019} replaced $[n]$ with a random subset of $[n]$ where each number is included with probability $p$ and proved asymptotic thresholds of $p$ for Breaker or/and Maker to win. However, unlike the van der Waerden games, the smallest $n$ such that Maker wins for the unbiased and deterministic Rado games are left unstudied. 

In this paper, we study the smallest positive integer $n$ such that Maker wins the Rado games on $[n]$ when $\mathcal{F}$ is the set of solutions to the equation
\begin{equation}\label{Equation:Fractional-Coefficients1}
x_1^{1/\ell}+\cdots+x_k^{1/\ell}=y^{1/\ell}
\end{equation}
where $k$ and $\ell$ are integers with $k\geq2$ and $\ell\neq0$. \cref{Equation:Fractional-Coefficients1} is connected with results in arithmetic Ramsey theory \cite{GrahamButler2015,LandmanRobertson2014}. In arithmetic Ramsey theory, a system of equations $E(x_1,\ldots,x_k,y)=0$ in variables $x_1,\ldots,x_k,y$ is called \textit{partition regular} if whenever $\mathbb{N}$ is partitioned into a finite number of classes, one of them contains a solution to $E(x_1,\ldots,x_k,y)=0$. In 1991, Lefmann \cite{Lefmann1991} proved that, among other things, \cref{Equation:Fractional-Coefficients1} is partition regular for all $\ell\in\mathbb{Z}\backslash\{0\}$. In the same year, Brown and R\"{o}dl \cite{BrownRodl1991} proved that if a system $E(x_1,\ldots,x_k,y)=0$ of homogeneous equations is partition regular, then the system $E(1/x_1,\ldots,1/x_k,1/y)=0$ is also partition regular.

To state our results, we first define the games we study in detail. Let $A\subseteq\mathbb{N}$ be a finite set and let $e(x_1,\ldots,x_k,y)=0$ be an equation in variables $x_1,\ldots,x_k,y$. The Maker-Breaker Rado games denoted $G(A,e(x_1,\ldots,x_k,y)=0)$ and $G^*(A,e(x_1,\ldots,x_k,y)=0)$ have the following rules:
\begin{itemize}
\item[(1)] Maker and Breaker take turns to select a number from $A$. Once a number is selected by a player, neither players can select that number again. Maker starts the game.
\item[(2)] Maker wins the $G(A,e(x_1,\ldots,x_k,y)=0)$ game if a collection of the numbers chosen by Maker form a solution to $e(x_1,\ldots,x_k,y)=0$ where $x_1,\ldots,x_k$ are \textit{not} necessarily distinct; and Maker wins the $G^*(A,e(x_1,\ldots,x_k,y)=0)$ game if a collection of the numbers chosen by Maker form a solution to $e(x_1,\ldots,x_k,y)=0$ where $x_1,\ldots,x_k$ are distinct.
\item[(3)] Breaker wins if Maker fails to occupy a solution to $e(x_1,\ldots,x_k,y)=0$.
\end{itemize}
We use the following shorter notations for games with \cref{Equation:Fractional-Coefficients1}:
\[
G([n],k,\ell):=G\left([n],x_1^{1/\ell}+\cdots+x_k^{1/\ell}=y^{1/\ell}\right)
\]
and
\[
G^*([n],k,\ell):=G^*\left([n],x_1^{1/\ell}+\cdots+x_k^{1/\ell}=y^{1/\ell}\right).
\]

We say that a player wins a game if there is a winning strategy which guarantees that this player wins no matter what the other player does. A \textit{winning strategy} is a set of instructions which tells the player what to do each round given what had been previously played by both players. Let $f(k,\ell)$ be the smallest positive integer $n$ such that Maker wins the $G([n],k,\ell)$ game and let $f^*(k,\ell)$ be the smallest positive integer $n$ such that Maker wins the $G^*([n],k,\ell)$ game.

For $\ell\geq1$, we are able to find exact formulas for $f(k,\ell)$ and $f^*(k,\ell)$.

\begin{theorem}\label{Theorem:Main-PositivePower-NotDistinct}
For all integers $k\geq2$ and $\ell\geq1$, we have $f(k,\ell)=(k+2)^\ell$.
\end{theorem}

\begin{theorem}\label{Theorem:Main-PositivePower-Distinct}
For all integers $k\geq2$ and $\ell\geq1$, we have $f^*(k,\ell)=(k^2+3)^\ell$.
\end{theorem}

Our proofs of \cref{Theorem:Main-PositivePower-NotDistinct,Theorem:Main-PositivePower-Distinct} involve showing that $f(k,1)=k+2$ and $f^*(k,1)=k^2+3$ using elementary combinatorial arguments, and that $f(k,\ell)\leq [f(k,1)]^\ell$ and $f^*(k,\ell)\leq [f^*(k,1)]^\ell$ using a result of Besicovitch \cite{Besicovitch1940} on the linear independence of integers with fractional powers.

For $\ell\leq-1$, our main results are the following:
\begin{theorem}\label{Theorem:Main-NegativePower-NotDistinct}
Let $k,\ell$ be integers with $\ell\leq-1$. Then $f(k,\ell)=[k+\Theta_k(1)]^{-\ell}$. More specifically, if $k\geq1/(2^{-1/\ell}-1)$, then $f(k,\ell)\geq(k+1)^{-\ell}$; and if $k\geq4$, then $f(k,\ell)\leq(k+2)^{-\ell}$.
\end{theorem}

\begin{theorem}\label{Theorem:Main-NegativePower-Distinct}
Let $k,\ell$ be integers with $\ell\leq-1$. Then $f^*(k,\ell)=[\exp(O_k(k\log k))]^{-\ell}$.
\end{theorem}

The proof of \cref{Theorem:Main-NegativePower-Distinct} involves showing that $f^*(k,-1)=\exp(O_k(k\log k))$ using a game theoretic variant of a theorem in arithmetic Ramsey theory by Brown and R\"{o}dl \cite{BrownRodl1991}.

Our results indicate that it is ``easier" to form a solution to \cref{Equation:Fractional-Coefficients1} strategically compared to their counterparts in arithmetic Ramsey theory. To illustrate this, let $R(k,\ell)$ be the smallest positive integer $n$ such that if $[n]$ is partitioned into two classes then one of them has a solution to \cref{Equation:Fractional-Coefficients1} with $x_1,\ldots,x_k$ not necessarily distinct, and let $R^*(k,\ell)$ be the smallest positive integer $n$ such that if $[n]$ is partitioned into two classes then one of them has a solution to \cref{Equation:Fractional-Coefficients1} with $x_1,\dots,x_k$ distinct. Note that if Maker and Breaker choose numbers in $[n]$, with $n\geq R(k,\ell)$ (respectively, $n\geq R^*(k,\ell)$), until there is no number left to choose, then the sets of numbers chosen by Maker and Breaker form a partition of $[n]$. If Maker does not win the game, then it means that the set of numbers chosen by Breaker contains a solution to \cref{Equation:Fractional-Coefficients1}. Since Maker goes first, by strategy stealing, Maker could follow Breaker's strategy and win the game. Therefore, we have $f(k,\ell)\leq R(k,\ell)$ and $f^*(k,\ell)\leq R^*(k,\ell)$. When $\ell\in\{-1,1\}$, some results on $R(k,\ell)$ and $R^*(k,\ell)$ are known.

For $\ell=1$, Beutelsapacher and Brestovansky \cite{BeutelspacherBrestovansky1982} proved that $R(k,1)=k^2+k-1$. The exact formula for $R^*(k,1)$ is not known, but Boza, Revuelta, and Sanz \cite{BRS2017} proved that, for $k\geq6$, $R^*(k,1)\geq(k^3+3k^2-2k)/2$. Hence, by \cref{Theorem:Main-PositivePower-NotDistinct,Theorem:Main-PositivePower-Distinct}, we have
\[
\lim_{k\to\infty}\frac{f(k,1)}{R(k,1)}=\lim_{k\to\infty}\frac{f^*(k,1)}{R^*(k,1)}=0.
\]

For $\ell=-1$, Myers and Parrish \cite{MyersParrish2018} calculated that $R(2,-1)=60$, $R(3,-1)=40$, $R(4,-1)=48$, and $R(5,-1)=39$; and the first author \cite{Gaiser2023} proved that $R(k,-1)\geq k^2$. So by \cref{Theorem:Main-NegativePower-NotDistinct}, we have
\begin{equation}\label{Equation:LimitUnitFraction}
    \lim_{k\to\infty}\frac{f(k,-1)}{R(k,-1)}=0.
\end{equation}

Unfortunately, we don't know a similar lower bound for $R^*(k,-1)$. However, we believe that Maker can still do better by selecting numbers strategically.
\begin{conjecture}
$\lim_{k\to\infty}f^*(k,-1)/R^*(k,-1)=0$.
\end{conjecture}

This paper is organized as follows. We first prove some preliminary results in \cref{Section:PreliminaryResults}. The next four sections are devoted to proving \cref{Theorem:Main-PositivePower-NotDistinct,Theorem:Main-PositivePower-Distinct,Theorem:Main-NegativePower-NotDistinct,Theorem:Main-NegativePower-Distinct}. In \cref{Section:AribitraryCoefficients}, we study Rado games for linear equations with arbitrary coefficients. We discuss some future research directions in \cref{Section:Conclusion}.
\subsection{Asymptotic Notation}
We use standard asymptotic notation. For functions $f(k)$ and $g(k)$, $f(k)=O_k(g(k))$ if there exist constants $K$ and $C$ such that $|f(k)|\leq C|g(k)|$ for all $k\geq K$; $f(k)=\Omega_k(g(k))$ if there exist constants $K'$ and $c$ such that $|f(k)|\geq c|g(k)|$ for all $k\geq K'$; $f(k)=\Theta_k(g(k))$ if $f(k)=O_k(g(k))$ and $f(k)=\Omega_k(g(k))$; and $f(k)=o_k(g(k))$ if $\lim_{k\to\infty}f(k)/g(k)=0$.

We remind the readers that, throughout this paper, \textbf{we only use asymptotic notation for functions of $k$} where $\ell$ is neither a parameter nor a constant.
\section{Preliminaries}\label{Section:PreliminaryResults}
We prove some results which will be used to prove \cref{Theorem:Main-PositivePower-NotDistinct,Theorem:Main-PositivePower-Distinct,Theorem:Main-NegativePower-NotDistinct,Theorem:Main-NegativePower-Distinct}. Our first result shows that the games for equations with radicals can be partially reduced to games for equation without radicals, i.e., $\ell=1$ or $\ell=-1$.
\begin{lemma}\label{Lemma:Linear-Radical}
Let $k$ and $\ell$ be integers with $k\geq2$ and $\ell\neq0$. If $\ell\geq 1$, then
\[
f(k,\ell)\leq[f(k,1)]^\ell\text{ and }f^*(k,\ell)\leq[f(k,1)]^\ell.
\]
If $\ell\leq-1$, then
\[
f(k,\ell)\leq[f(k,-1)]^{-\ell}\text{ and }f^*(k,\ell)\leq[f(k,-1)]^{-\ell}.
\]
\end{lemma}
\begin{proof}
We prove that if $\ell\geq1$, then $f(k,\ell)\leq[f(k,1)]^\ell$. The other inequalities can be proved similarly.

Write $M=f(k,1)$ and let $\mathcal{M}$ be a Maker's winning strategy for the $G([M],k,1)$ game. Notice that if $(x_1,\ldots,x_k,y)=(a_1,\ldots,a_k,b)$ is a solution to $x_1+\cdots+x_k=y$, then $(x_1,\ldots,x_k,y)=(a_1^\ell,\ldots,a_k^\ell,b^\ell)$ is a solution to $x_1^{1/\ell}+\cdots+x_k^{1/\ell}=y^{1/\ell}$.

For $i=1,2,\ldots$, let $m_i\in[M^\ell]$ be the number chosen by Maker and let $b_i\in[M^\ell]$ be the number chosen by Breaker in round $i$. We define a strategy for Maker recursively. We note that Maker focuses on the set $\{1^\ell,2^\ell,\ldots,M^\ell\}$ in this strategy. In round 1, if $\mathcal{M}$ tells Maker to choose $a_1$ for the $G([M],k,1)$ game, then set $m_1=a_1^\ell$. If $b_1=z_1^\ell$ for some $z_1\in[M]$, then set $b_1'=z_1$; otherwise, arbitrarily set $b_1'$ equal to some number in $M\backslash\{a_1\}$. In round $i\geq2$, given $a_1,a_2,\ldots,a_{i-1},b_1',b_2',\ldots,b_{i-1}'$, if $\mathcal{M}$ tells Maker to choose $a_i$, then set $m_i=a_i$. This is possible because $\mathcal{M}$ is a winning strategy. If $b_i=z_i^\ell$ for some $z_i\in[M]$, then set $b_i'=z_i$; otherwise, arbitrarily set $b_i'$ equal to some number in $M\backslash\{a_1,a_2,\ldots,a_{i-1},a_i,b_1',b_2',\ldots,b_{i-1}'\}$.

Now since $\mathcal{M}$ is a winning strategy, there exists $t$ such that $\{a_1,a_2,\ldots,a_t\}$ has a solution to $x_1+\cdots+x_k=y$. Hence $\{m_1,m_2,\ldots,m_t\}=\{a_1^\ell,a_2^\ell,\ldots,a_t^\ell\}$ has a solution to $x_1^{1/\ell}+\cdots+x_k^{1/\ell}=y^{1/\ell}$. Therefore, Maker wins the $G([M^\ell],k,\ell)$ game.
\end{proof}

\cref{Theorem:Main-PositivePower-NotDistinct,Theorem:Main-PositivePower-Distinct} indicate that these inequalities in \cref{Lemma:Linear-Radical} are actually equalities when $\ell\geq 2$. This is due to a result of Besicovitch \cite{Besicovitch1940}. To state this result, we first need the following definition.

\begin{definition}
Let $a\in\mathbb{N}\backslash\{1\}$. We say that $a$ is \textit{power-$\ell$ free} if $a=b^\ell c$, with $b,c\in\mathbb{N}$, implies $b=1$. 
\end{definition}

\begin{theorem}[Besicovitch \cite{Besicovitch1940}]\label{Theorem:LinearIndependence}
For all positive integers $\ell\geq2$, the set
\[
A(\ell):=\{a^{1/\ell}:a\in\mathbb{N}\backslash\{1\}\text{ and }a \text{ is power-}\ell \text{ free}\}
\]
is linearly independent over $\mathbb{Z}$. That is, if $a_1,\ldots,a_m\in A(\ell)$ and $c_1,\ldots,c_m\in\mathbb{N}$ satisfy $c_1a_1+\cdots+c_ma_m=0$, then $c_1=\cdots=c_m=0$.
\end{theorem}

Besicovitch \cite{Besicovitch1940} actually provided an elementary proof of a stronger result, but \cref{Theorem:LinearIndependence} is enough for our purposes. For interested readers, we note that Richards \cite{Richards1974} proved a similar result to the one in \cite{Besicovitch1940}, but using Galois theory instead. A direct consequence of \cref{Theorem:LinearIndependence} is the following result which will be used in proving \cref{Theorem:Main-PositivePower-NotDistinct,Theorem:Main-PositivePower-Distinct}.

\begin{corollary}\label{Corollary:SolutionsFractional}
Let $k\geq2$ and $\ell\geq1$ be integers. The solutions to $x_1^{1/\ell}+\cdots+x_k^{1/\ell}=y^{1/\ell}$ are of the form $(x_1,\ldots,x_k,y)=(ca_1^\ell,\ldots,ca_k^\ell,cb^\ell)$ where $a_1,\ldots,a_k,b,c\in\mathbb{N}$, $a_1+\cdots+a_k=b$, and $c$ is power-$\ell$ free.
\end{corollary}
\begin{proof}
Suppose that $\alpha_1,\ldots,\alpha_k,\beta\in\mathbb{N}$ satisfy
\[
\alpha_1^{1/\ell}+\cdots+\alpha_k^{1/\ell}=\beta^{1/\ell}.
\]
We write $\alpha_i=c_ia_i^\ell$ for all $i=1,...,k$, and $\beta=db^\ell$ where $a_1,\ldots,a_k,c_1,\ldots,c_k,b,d\in\mathbb{N}$ and $c_1,\ldots,c_k,d$ are power-$\ell$ free. Then we have
\begin{equation}\label{Equation:AFractionalSolution}
a_1c_1^{1/\ell}+\cdots+a_kc_k^{1/\ell}-bd^{1/\ell}=0.
\end{equation}

We first show that $c_1=\cdots=c_k=d$. Suppose, for a contradiction, that $c_1,\ldots,c_k,d$ are not all the same. We split this into two cases.

Case 1: $d\neq c_i$ for all $i\in[k]$. After combining terms with same $\ell$-th roots, the left-hand side of \cref{Equation:AFractionalSolution} has at least two terms where one of them is $-bd^{1/\ell}$. Now by \cref{Theorem:LinearIndependence}, $b=0$ which is a contradiction.

Case 2: $d=c_i$ for some $i\in[k]$. Then there exists $j\in[k]\backslash\{i\}$ such that $c_j\neq c_i$. After combining terms with same $\ell$-th roots, the left-hand side of \cref{Equation:AFractionalSolution} has a term with $c_j^{1/\ell}$. This is because all the terms with $c_j^{1/\ell}$ contain only positive coefficients. By \cref{Theorem:LinearIndependence}, the coefficient of $c_j^{1/\ell}$ is zero after combining like terms. But this is impossible because the coefficient of $c_j^{1/\ell}$ is the sum of a subset of $\{a_1,...,a_k\}$ consisting only positive integers.

Hence we have $c_1=\cdots=c_k=d$. Therefore, $a_1+\cdots+a_k=b$.
\end{proof}
We note that Newman \cite{Newman1981} proved \cref{Corollary:SolutionsFractional} for the case $k=2$ without using \cref{Theorem:LinearIndependence}.

Next, we prove a game theoretic variant of a result by Brown and R\"{o}dl \cite[Theorem 2.1]{BrownRodl1991}. We note that an equation $e(x_1,\ldots,x_k,y)=0$ is \textit{homogeneous} if whenever $(x_1,\ldots,x_k,y)=(a_1,\ldots,a_k,b)$ is a solution to $e(x_1,\ldots,x_k,y)=0$, for all $m\in\mathbb{N}$, $(x_1,\ldots,x_k,y)=(ma_1,\ldots,ma_k,mb)$ is a also a solution to $e(x_1,\ldots,x_k,y)=0$.

\begin{theorem}\label{Theorem:GameBrownRodl}
Let $A$ be a finite subset of $\mathbb{N}$, $L$ the least common multiple of $A$, $k\in\mathbb{N}$, and $e(x_1,\ldots,x_k,y)=0$ a homogeneous equation. If Maker wins the $G(A,e(x_1,\ldots,x_k,y)=0)$ game, then Maker wins the $G([L],e(1/x_1,\ldots,1/x_k,1/y)=0)$ game. Similarly, if Maker wins the \\$G^*(A,e(x_1,\ldots,x_k,y)=0)$ game, then Maker wins the $G^*([L],e(1/x_1,\ldots,1/x_k,1/y)=0)$ game.
\end{theorem}
\begin{proof}
Suppose that Maker wins the $G(A,e(x_1,\ldots,x_k,y)=0)$ game. Let $\mathcal{M}$ be a Maker's winning strategy. We consider the following Maker's strategy for the $G([L],e(1/x_1,\ldots,x_k,1/y)=0)$ game. In round 1, if $\mathcal{M}$ tells Maker to choose $m_1$ for the $G(A,e(x_1,\ldots,x_k,y)=0)$ game, then Maker chooses $L/m_1\in\{1,\ldots,L\}$. The rest of the strategy is defined inductively. For all rounds $i$, let $L/b_i$ be the number chosen by Breaker and $L/m_i$ be the number chosen by Maker where $m_i\in\{1,\ldots,L\}$. If $b_i\in A$, then we set $b_i'=b_i$; if $b_i\notin A$, then arbitrarily set $b_i'$ equal to some number in $A\backslash\{m_1,\ldots,m_i,b_1',\ldots,b_{i-1}'\}$. For all rounds $i\geq2$, given $\{m_1,\ldots,m_{i-1},b_1',\ldots,b_{i-1}'\}$, if $\mathcal{M}$ tells Maker to choose $m_i$ for the $G(A,e(x_1,\ldots,x_k,y)=0)$ game, then Maker chooses $L/m_i$ for the $G([L],e(1/x_1,\ldots,1/x_k,1/y)=0)$ game. This process is possible because $\mathcal{M}$ is a winning strategy.

Since $\mathcal{M}$ is a winning strategy, in some round $t$, there exists a subset $\{a_1,\ldots,a_s\}$ of $\{m_1,\ldots,m_t\}$ which form a solution to $e(x_1,\ldots,x_k,y)=0$. By homogeneity, $\{L/a_1,\ldots,L/a_{s}\}$ form a solution to $e(1/x_1,\ldots,1/x_k,1/y)=0$. So Maker wins the $G([L],e(1/x_1,\ldots,1/x_k,1/y)=0)$ game. 

The case for the $G^*([L],e(1/x_1,\ldots,1/x_k,1/y)=0)$ game can be proved in a similar way.
\end{proof}

The key feature of \cref{Theorem:GameBrownRodl} is that one can choose a set $A$ whose least common multiple $L$ is small. This was not used by Brown and R\"{o}dl \cite[Theorem 2.1]{BrownRodl1991}. For interested readers, we note that the first author \cite{Gaiser2023} recently improved a quantitative result by Brown and R\"{o}dl \cite[Theorem 2.5]{BrownRodl1991} with the help of this observation.

Finally, we also need the following definitions.

\begin{definition}
Given $m\in\mathbb{N}$ mutually disjoint subsets $\{s_1,t_1\}$, $\{s_2,t_2\}$, \ldots, $\{s_m,t_m\}$ of $\mathbb{N}$ with size $2$, the \textit{pairing strategy} over those disjoint subsets for a player is defined as follows: if their opponent chooses $s_i$ for some $i=1,2,\ldots,m$, then this player chooses $t_i$.
\end{definition}

\begin{definition}
Let $k\geq2$ be an integer and $a_1x_1+\cdots+a_kx_k=y$ a linear equation. Suppose, at some point of the $G^*([n],a_1x_1+\cdots+a_kx_k=y)$ game, Maker has claimed a set $A$ of at least $k$ integers. We call $a_1\alpha_1+\cdots+a_k\alpha_k$ a \textit{$k$-sum} for any $k$ distinct integers $\alpha_1,\ldots,\alpha_k\in A$.
\end{definition}
\section{Proof of \cref{Theorem:Main-PositivePower-NotDistinct}}\label{Section:Proof-MainTheorem-PositiveNotDistinct}
Let $k,\ell$ be integers with $k\geq2$ and $\ell\geq1$. We first show that $f(k,1)=k+2$.
\begin{lemma}\label{Lemma:Power1-NonDistinct}
For all integers $k\geq2$, we have $f(k,1)=k+2$.
\end{lemma}
\begin{proof}
We first show that Maker wins the $G([k+2],k,1)$ game. Note that this will be proved in more full generality later in \cref{Theorem:AritraryCoefficients-NonDistinct}.

Case 1: $k=2$. Maker starts by choosing 2. Since $2+2=4$ and $1+1=2$, Maker wins the game in the next round by choosing either $1$ or $4$, whichever is available.

Case 2: $k>2$.  Maker starts by selecting 1. Notice that \[\underbrace{1+1+\cdots+1}_k=k\cdot1=k,\] \[\underbrace{1+1+\cdots+1}_{k-1}+2=(k-1)\cdot1+2=k+1,\] and \[\underbrace{1+1+\cdots+1}_{k-2}+2+2=(k-2)\cdot1+2\cdot2=k+2.\] If Breaker chooses $k$ in the first round, then Maker chooses 2 in round $2$ and wins the game in round $3$ by choosing either $k+1$ or $k+2$. If Breaker does not choose $k$ in round $1$, then Maker can win the game in round $2$ by choosing $k$.

Now we show that Breaker wins the $G([k+1],k,1)$ game. When $\ell=1$, the only possible solutions to \cref{Equation:Fractional-Coefficients1} in $\{1,\ldots,k+1\}$ are \[(x_1,x_2,\ldots,x_{k-1},x_k,y)=(1,1,\ldots,1,1,k)\] and \[(x_1,x_2,\ldots,x_{k-1},x_k,y)=(1,1,\ldots,1,2,k+1).\] If $k=2$, then Breaker wins the game by the pairing strategy over $\{1,2\}$. If $k\geq3$, then Breaker wins the game by the pairing strategy over $\{1,k\}$ and $\{2,k+1\}$.
\end{proof}

By \cref{Lemma:Linear-Radical,Lemma:Power1-NonDistinct}, we have $f(k,\ell)\leq[f(k,1)]^\ell=(k+2)^\ell$. It remains to show that $f(k,\ell)\geq (k+2)^\ell$. This is true for $\ell=1$ by \cref{Lemma:Power1-NonDistinct}. So we assume $\ell\geq2$. It suffices to show that Breaker wins the $G\left([(k+2)^\ell-1],k,\ell\right)$ game. To do this, we need a result on solutions to $x_1^{1/\ell}+\cdots+x_k^{1/\ell}=y^{1/\ell}$ in $\{1,2,\ldots,(k+2)^\ell-1\}$.

\begin{lemma}\label{Lemma:SolutionsTolPowerNonDistinct}
The only solutions to $x_1^{1/\ell}+\cdots+x_k^{1/\ell}=y^{1/\ell}$ in $\{1,2,\ldots,(k+2)^\ell-1\}$ are \[(x_1,\ldots,x_{k-2},x_{k-1},x_k,y)=(a,\ldots,a,a,a,ak^\ell),\]
and
\[
(x_1,\ldots,x_{k-2},x_{k-1},x_k,y)=(b,\ldots,b,b,b2^\ell,b(k+1)^\ell),
\]
where $a,b\in\{1,2,\ldots,2^{\ell}-1\}$ and are power-$\ell$ free. 
\end{lemma}
\begin{proof}
By \cref{Corollary:SolutionsFractional}, the only solutions to $x_1^{1/\ell}+\cdots+x_k^{1/\ell}=y^{1/\ell}$ in $\mathbb{N}$ are $(x_1,\ldots,x_k,y)=(c\alpha_1^\ell,\ldots,c\alpha_k^\ell,c\beta^\ell)$ where $\alpha_1,\ldots,\alpha_k,\beta,c\in\mathbb{N}$, $\alpha_1+\cdots+\alpha_k=\beta$, and $c$ is power-$\ell$ free. Restricted to the set $\{1,2,\ldots,(k+2)^\ell-1\}$, we must have $c\alpha_1^\ell,\ldots,c\alpha_k^\ell,c\beta^\ell\leq(k+2)^\ell-1$. It follows that $\alpha_1^\ell,\ldots,\alpha_k^\ell\in\{1^\ell,2^\ell,\ldots,(k+1)^\ell\}$ and hence $\alpha_1,\ldots,\alpha_k,\beta\leq k+1$. So $\alpha_1,\ldots,\alpha_k,\beta$ form a solution to $x_1+\cdots+x_k=y$ in $\{1,2,\ldots,k+1\}$. Since the only solutions to $x_1+\cdots+x_k=y$ in $\{1,2,\ldots,k+1\}$ are 
\[(x_1,\ldots,x_{k-1},x_k,y)=(1,\ldots,1,1,k),\] 
and
\[(x_1,\ldots,x_{k-1},x_k,y)=(1,\ldots,1,2,k+1),\] 
we have either \[(\alpha_1,\ldots,\alpha_{k-1},\alpha_k,\beta)=(1,\ldots,1,1,1,k)\] or \[(\alpha_1,\ldots,\alpha_{k-1},\alpha_k,\beta)=(1,\ldots,1,2,k+1).\]
Now since $c\beta^\ell\leq(k+2)^\ell-1$, we have
\[
c\leq\frac{(k+2)^\ell-1}{\beta^\ell}\leq\frac{(k+2)^\ell-1}{k^\ell}<\left(1+\frac{2}{k}\right)^\ell\leq 2^\ell.
\]
Hence $c\in\{1,2,\ldots,2^\ell-1\}$.
\end{proof}

If $k=2$, then Breaker wins the game by the pairing strategy over the sets $\{a,a2^\ell\}$ where $a\in\{1,2,\ldots,2^\ell-1\}$. If $k\geq3$, then Breaker wins the game by the pairing strategy over the sets $\{a,ak^\ell\}$ and $\{b2^\ell,b(k+1)^\ell\}$ where $a,b\in\{1,2,\ldots,2^\ell-1\}$. In these pairing strategies, if Maker chooses some $a$ or $b2^\ell$ so that $ak^\ell>(k+2)^\ell-1$ or $b(k+1)^\ell>(k+2)^\ell-1$, then Breaker arbitrarily chooses an available number in $\{1,2,\ldots,(k+2)^\ell-1\}$.
\section{Proof of \cref{Theorem:Main-PositivePower-Distinct}}\label{Section:Proof-MainTheorem-PositiveDistinct}
Let $k,\ell$ be integers with $k\geq2$ and $\ell\geq1$. We first establish that $f^*(k,1)=k^2+3$.
\begin{lemma}\label{Lemma:Power1Distinct-Upper}
For all integers $k\geq2$, we have $f^*\left(k,1\right)\leq k^2+3$.
\end{lemma}
\begin{proof}
It suffices to show that Maker wins the $G^*([k^2+3],k,1)$ game. For $i=1,2,...,\lceil n/2\rceil$, let $m_i$ denote the number selected by Maker in round $i$. For $j=1,2,\ldots,\lfloor n/2\rfloor$, let $b_j$ denote the number selected by Breaker in round $j$.

We first consider the case that $k=2$. Then $k^2+3=7$. Maker starts by choosing $m_1=1$. Then no matter what $b_1$ is, there are three consecutive numbers in $\{2,3,4,5,6,7\}$ available to Maker, say $\{a,b,c\}$. Maker sets $m_2=b$. Notice that $1+a=b$ and $1+b=c$. Since Breaker can only choose one of $a$ and $c$, Maker wins in round $3$ by setting $m_3=a$ or $m_3=c$.

Now suppose $k=3$. Then $k^2+3=12$. Maker starts by choosing $m_1=1$. We have 4 cases based on Breaker's choices.

Case 1: If $b_1\neq 2$, then Maker chooses $m_2=2$. Suppose Breaker has selected $b_2$. Now consider the 3-term arithmetic progressions of difference $m_1+m_2=3$:
\[
\{3,6,9\}, \{4,7,10\},~\text{and}~\{5,8,11\}.
\]
At the start of round 3, Breaker has chosen two numbers and hence one of these 3-term arithmetic progressions is available to Maker. Maker can set $m_3$ equal to the middle number of the available 3-term arithmetic progression and win the game in round $4$ by choosing either the smallest or the largest number of the same 3-term arithmetic progression.

Case 2: If $b_1=2$, then Maker chooses $m_2=3$. Suppose $b_2\neq4,8,12$. Since $\{4,8,12\}$ is a 3-term arithmetic progression of difference $m_1+m_2=4$, Maker can set $m_3=8$ and win the game in round $4$ by choosing either $4$ or $12$.

Case 3: If $b_1=2$, then Maker chooses $m_2=3$. Suppose $b_2=4$ or $8$. Then Maker sets $m_3=5$. If $b_3\neq9$, then Maker sets $m_4=9$. Since $m_1+m_2+m_3=1+3+5=9=m_4$, Maker wins the game. Suppose $b_3=9$. Then Maker sets $m_4=6$. Since $m_1+m_2+m_4=1+3+6=10$ and $m_1+m_3+m_4=1+5+6=12$, Maker wins in round 5 by choosing either 10 or 12.

Case 4: If $b_1=2$, then Maker chooses $m_2=3$. Suppose $b_2=12$. Then Maker sets $m_3=4$. If $b_3\neq8$, then Maker sets $m_4=8$. Since $m_1+m_2+m_3=1+3+4=8=m_4$, Maker wins the game. Suppose $b_3=8$. Then Maker sets $m_4=5$. Since $m_1+m_2+m_4=1+3+5=9$ and $m_1+m_3+m_4=1+4+5=10$, Maker wins in round 5 by choosing either 9 or 10.

Finally, we consider that $k\geq4$. We start with an observation.

\textbf{Claim 1.} Since $k\geq4$, all the $k$-sums are at least
\[
\sum_{i=1}^ki=\frac{1}{2}k^2+\frac{1}{2}k>2k.
\]

We prove that Maker can win with the following strategy: if a $k$-sum is available to Maker, then Maker chooses the $k$-sum and win the game; otherwise Maker selects the smallest number available. By this strategy, Maker will choose the smallest numbers possible for the first $k$ rounds and the smallest $k$-sum is $m_1+\cdots+m_k$.

\textbf{Claim 2.} $m_i\leq2i-1$ for $i=1,...,k$. Indeed, at the start of round $i$, Maker and Breaker have together chosen $2(i-1)=2i-2$ numbers. Hence, one of the numbers in $\{1,2,\ldots,2i-1\}$ is still available to Maker. So by Maker's strategy, we have $m_i\leq 2i-1$.

By Claim 2, we have
\[
\sum_{i=1}^km_i\leq1+3+\cdots+2k-1=k^2\leq k^2+3.
\]

If Breaker didn't choose $m_1+\cdots+m_k$ during the first $k$ rounds, then Maker chooses $m_1+\cdots+m_k$ in round $k+1$ and wins the game.

Now suppose that Breaker has selected $m_1+\cdots+m_k$ during the first $k$ rounds. Consider the middle of round $k+1$ when Maker has chosen $k+1$ numbers but Breaker has only chosen $k$ numbers where $s$, $1\leq s\leq k$, of them are $k$-sums. Since there are $2k+1$ numbers in $\{1,2,\ldots,2k+1\}$ and Breaker has chosen only $k$ numbers, we have $m_{k+1}\leq2k+1$ by Maker's strategy. Since $m_1,\ldots,m_{k+1}$ are distinct, the total number of $k$-sums is ${k+1\choose k}=k+1$. 

\textbf{Claim 3.} If Breaker has chosen $s$ $k$-sums during the first $k$ rounds and one of them is $\sum_{i=1}^km_i$, then $m_{k+1-s+j}\leq2(k+1-s+j)-1-j=2(k+1-s)+j-1$ for $j=1,2,\ldots,s$. By Claim 1, the $k$-sums are greater than $2k$. So if Breaker has chosen $s$ $k$-sums, then Breaker has chosen at most $k-s$ numbers in $\{1,2,\ldots,2k-s+1\}$. By Maker's strategy, Maker has chosen $k+1$ numbers in $\{1,2,\ldots,2k-s+1\}$. If $s=1$, then we have $m_{k+1}\leq 2k$ and the claim is true. If $s>1$, then by Maker's strategy, we have $m_{k+1}>m_k>\cdots>m_{k+1-s+1}$. Since $m_{k+1},\ldots,m_{k+1-s+1}\in\{1,2,\ldots,2k-s+1\}$, the claim is also true.

Now we split it into two cases based on the value of $s$ and what Breaker chooses in round $k+1$.

Case 1: $1\leq s\leq k-1$ or $s=k$ and Breaker does not choose a $k$-sum in round $k+1$. Then Breaker will have chosen at most $k$ $k$-sums at the beginning of round $k+2$. By Claim 2 and Claim 3, at the beginning of round $k+2$, there exists an unclaimed $k$-sum whose value is at most
\[
\begin{split}
\sum_{i=1}^{k+1-s-2}m_i+\sum_{i=k+1-s}^{k+1}m_i\leq&\sum_{i=1}^{k+1-s-2}(2i-1)+\sum_{j=0}^s[2(k+1-s)+j-1]\\=&(k-s-1)^2+(s+1)2(k+1-s)+\frac{s(s-1)}{2}-1\\=&k^2-\frac{1}{2}s^2+\frac{3}{2}s+2\leq k^2+3.
\end{split}
\]
Hence Maker chooses this $k$-sum in round $k+2$ and wins the $G^*([k^2+3],k,1)$ game.

Case 2: $s=k$ and Breaker chooses a $k$-sum in round $k+1$. In this cases, at the end of round $k+1$, Breaker has chosen all possible $k$-sums from $\{m_1,\ldots,m_{k+1}\}$. By Claim 1, the $k$-sums are greater than $2k$. Since $k+2\leq2k$ for $k\geq2$, Breaker didn't choose any number in $\{1,2,\ldots,k+2\}$. So $m_i=i$ for $i=1,2,\ldots,k+2$. Notice that the largest $k$-sum before round $k+2$ is
\[
\sum_{i=2}^{k+1}m_i=\sum_{i=1}^{k+1}i=\frac{(k+1)(k+2)}{2}-1=\frac{1}{2}k^2+\frac{3}{2}k.
\]

Setting $m_{k+2}=k+2$, Maker now has two larger $k$-sums which are untouched by Breaker:
\[
m_{k+2}+\sum_{i=2}^{k}m_i=k+2+\frac{k(k+1)}{2}-1=\frac{1}{2}k^2+\frac{3}{2}k+1
\]
and
\[
m_{k+1}+m_{k+2}+\sum_{i=2}^{k-1}m_i=k+1+k+2+\frac{(k-1)k}{2}-1=\frac{1}{2}k^2+\frac{3}{2}k+2.
\]
Since $k\geq4$, we have
\[
k^2+3\geq\frac{1}{2}k^2+\frac{3}{2}k+2.
\]
Hence Maker can win the $G^*([k^2+3],k,1)$ game in round $k+3$.
\end{proof}

\begin{lemma}\label{Lemma:Power1Distinct-Lower}
For all integers $k\geq2$, we have $f^*\left(k,1\right)\geq k^2+3$.
\end{lemma}
\begin{proof}
It suffices to show that Breaker wins the $G([k^2+2],k,1)$ game. For $i=1,2,\ldots,\lceil n/2\rceil$, let $m_i$ denote the number selected by Maker in round $i$. For $j=1,2,\ldots,\lfloor n/2\rfloor$, let $b_j$ denote the number selected by Breaker in round $j$.

We first consider $k=2$. Then $k^2+2=2^2+2=6$. If $m_1=1$, then Breaker chooses $b_1=4$. Now Breaker wins by the pairing strategy over $\{2,3\}$ and $\{5,6\}$. If $m_1\neq1$, then Breaker chooses $b_1=1$. Now there are only two solutions available to Maker: $2+3=5$ and $2+4=6$. There are three cases:

Case 1: $m_1=2$. Then Breaker wins by the pairing strategy over $\{3,5\}$ and $\{4,6\}$.

Case 2: $m_1\neq1,2$, $b_1=1$, $m_2=2$. Then Breaker wins by the pairing strategy over $\{3,5\}$ and $\{4,6\}$.

Case 3: $m_1\neq1,2$, $b_1=1$, $m_2\neq2$. Then by choosing $b_2=2$, Breaker wins because the smallest numbers now available to Maker are $3$ and $4$, and $3+4=7>6$.

Now we consider $k\geq3$. Notice that we have $k^2-1\geq 2k+2$ when $k\geq3$. We will prove that Breaker wins with the following strategy: 
\begin{itemize}
    \item[(1)] in each round $i\in[k-1]$, Breaker chooses smallest number available;
    \item[(2)] and in round $k$, if there is an unclaimed number in $[2k-2]$, then Breaker chooses the unclaimed number; otherwise, Breaker's strategy depends on the sum of the numbers in $[2k-2]$ claimed by Maker, which is denoted by $S$:
        \begin{itemize}
        \item[•] If $S=(k-1)^2+3$, then Breaker chooses smallest numbers possible.
        \item[•] If $S=(k-1)^2+2$, then Breaker plays the pairing strategy over $\{2k-1,k^2+2\}$.
        \item[•] If $S=(k-1)^2+1$, then Breaker plays the pairing strategy over $\{2k-1,k^2+1\}$ and $\{2k,k^2+2\}$.
        \item[•] If $S=(k-1)^2$, then Breaker plays the pairing strategy over $\{2k-1,k^2\}$, $\{2k,k^2+1\}$, and $\{2k+1,k^2+2\}$.
        \end{itemize}
\end{itemize}

Let $a_1<a_2<a_3<\cdots<a_{s}$ with $s\leq\lceil n/2\rceil$ be the numbers chosen by Maker when the game ends.

\textbf{Claim 1:} $a_i\geq2i-1$ for $i=1,2,\ldots,k$, $a_{k+1}\geq2k$, and $a_{k+2}\geq2k+1$. Since $a_i\geq1=2\cdot1-1$, this is true for $i=1$. Now consider $2\leq i\leq k$. By Breaker's strategy, Breaker can select at least $i-1$ numbers in $\{1,\ldots,2(i-1)\}$. So Maker can select at most $i-1$ numbers in $\{1,\ldots,2(i-1)\}$. Hence $a_i\geq2(i-1)+1=2i-1$.

\textbf{Claim 2:} If $a_{k-1}>2k-2$, then Breaker wins. If this happens, then $a_{k-1}\geq2k-1$ and $a_k\geq2k$. Hence the smallest $k$-sum possible for Maker is
\[
\sum_{i=1}^{k}a_i\geq2k-1+2k+\sum_{i=1}^{k-2}(2i-1)=2k-1+2k+(k-2)^2=k^2+3>k^2+2
\]
and hence Breaker wins.

\textbf{Claim 3:} The smallest $k$-sum possible for Maker is $\sum_{i=1}^ka_i\geq\sum_{i=1}^k(2i-1)=k^2$. So Maker needs one of $k^2$, $k^2+1$, and $k^2+2$ to win.

\textbf{Claim 4:} If a $k$-sum does not contain all $\{a_1,...,a_{k-1}\}$, then Breaker wins. Indeed, if a $k$-sum does not contain all of $\{a_1,\ldots,a_{k-1}\}$, then the $k$-sum is at least
\[
a_k+a_{k+1}+\sum_{i=1}^{k-2}a_i\geq2k-1+2k+(k-2)^2=k^2+3>k^2+2.
\]

We first suppose that after Maker has chosen $m_1,\ldots,m_k$, there is an unclaimed number in $[2k-2]$. In this case, Breaker sets $b_k$ equal to some number in $[2k-2]$. Now Breaker has chosen $k$ numbers in $[2k-2]$ which implies that Maker can choose at most $k-2$ numbers in $[2k-2]$. Hence $a_{k-1}>2k-2$. By Claim 2, Breaker wins.

Now assume that all the numbers in $[2k-2]$ are claimed in the middle of round $k$ when Breaker has chosen $k$ numbers and Breaker has chosen $k-1$ numbers. In this case, we must have $a_1,\ldots,a_{k-1}\in[2k-2]$ and hence $\sum_{i=1}^{k-1}a_i=S$. We consider the solutions to $x_1+\cdots+x_k=y$, where $x_1,\ldots,x_k$ are distinct, such that Breaker has not occupied any number in them. By Claim 4, if a $k$-sum does not contain all numbers in $\{a_1,\ldots,a_{k-1}\}$, then Breaker wins. So we have the following cases:

Case 1: If $S=\sum_{i=1}^{k-1}a_i=(k-1)^2$, then there are three solutions to $x_1+\cdots+x_k=y$, where $x_1,\ldots,x_k$ are distinct, such that Breaker has not occupied any number in them: $\{a_1,\ldots,a_{k-1}, 2k-1,k^2\}$, $\{a_1,\ldots,a_{k-1},2k,k^2+1\}$, and $\{a_1,\ldots,a_{k-1},2k+1,k^2+2\}$. This is because if $S=\sum_{i=1}^{k-1}a_i=(k-1)^2$, then
\[
a_k+\sum_{i=1}^{k-1}a_i\geq2k-1+(k-1)^2=k^2,
\]
\[
a_{k+1}+\sum_{i=1}^{k-1}a_i\geq2k+(k-1)^2=k^2+1,
\]
\[
a_{k+2}+\sum_{i=1}^{k-1}a_i\geq2k+1+(k-1)^2=k^2+2,
\]
and
\[
a_{s}+\sum_{i=1}^{k-1}a_i\geq2k+1+1+(k-1)^2=k^2+3>k^2+2
\]
for $s\geq k+3$.

Case 2: If $S=\sum_{i=1}^{k-1}a_i=(k-1)^2+1$, then there are two solutions to $x_1+\cdots+x_k=y$, where $x_1,\ldots,x_k$ are distinct, such that Breaker has not occupied any number in them: $\{a_1,\ldots,a_{k-1},k^2+1\}$ and $\{a_1,\ldots,a_{k-1},a_{k+1},k^2+2\}$. This is because if $S=\sum_{i=1}^{k-1}a_i=(k-1)^2+1$, then
\[
a_k+\sum_{i=1}^{k-1}a_i\geq2k-1+(k-1)^2+1=k^2+1,
\]
\[
a_{k+1}+\sum_{i=1}^{k-1}a_i\geq2k+(k-1)^2+1=k^2+2,
\]
and
\[
a_{s}+\sum_{i=1}^{k-1}a_i\geq2k+1+(k-1)^2+1=k^2+3>k^2+2
\]
for $s\geq k+2$.

Case 3: If $S=\sum_{i=1}^{k-1}a_i=(k-1)^2+2$, then there is only one solution to $x_1+\cdots+x_k=y$, where $x_1,\ldots,x_k$ are distinct, such that Breaker has not occupied any number in them: $\{a_1,\ldots,a_k,k^2+2\}$. This is because if $S=\sum_{i=1}^{k-1}a_i=(k-1)^2+2$, then
\[
a_k+\sum_{i=1}^{k-1}a_i\geq2k-1+(k-1)^2+2=k^2+2,
\]
and
\[
a_{s}+\sum_{i=1}^{k-1}a_i\geq2k+(k-1)^2+2=k^2+3>k^2+2
\]
for $s\geq k+1$.

In Case 1, Breaker uses the pairing strategy  over$\{2k-1,k^2\}$, $\{2k,k^2+1\}$, and $\{2k+1,k^2+2\}$. Since these sets are pairwise disjoint, Breaker wins. Similarly, in Case 2, Breaker uses the pairing strategy over $\{2k-1,k^2+1\}$ and $\{2k,k^2+2\}$; and in Case 3, Breaker uses the pairing strategy over $\{2k-1,k^2+2\}$.
\end{proof}
By \cref{Lemma:Linear-Radical,Lemma:Power1Distinct-Upper,Lemma:Power1Distinct-Lower}, we have $f^*(k,\ell)\leq[f^*(k,1)]^\ell=(k^2+3)^\ell$. It remains to show that $f^*(k,\ell)\geq(k^2+3)^\ell$ for all $\ell\geq2$. To do this, it suffices to show that Breaker wins the $G([(k^2+3)^\ell-1],k,\ell)$ game. For all $c\in\{1,2,\ldots,2^\ell-1\}$, let 
\[
A(c)=\{c\cdot1^\ell,c\cdot2^\ell,\ldots,c\cdot(k^2+2)^\ell\}\cap\{1,2,\ldots,(k^2+3)^\ell-1\}.
\]
Notice that if $c,c'\in\{1,2,\ldots,2^\ell-1\}$ with $c\neq c'$, then $A(c)\cap A(c')=\emptyset$. By \cref{Corollary:SolutionsFractional}, every solution to $x_1^{1/\ell}+\cdots+x_k^{1/\ell}=y^{1/\ell}$, with $x_1,\ldots,x_k$ distinct, in $\{1,2,\ldots,(k^2+3)^\ell-1\}$ belongs to $A(c)$ for some $c\in\{1,2,\ldots,2^{\ell-1}\}$.

Let $\mathcal{B}$ be a Breaker's winning strategy for the $G^*([k^2+2],k,1)$ game. We define a Breaker's strategy for the $G([(k^2+3)^\ell-1],k,\ell)$ game recursively. For rounds $i=1,2,\ldots$, let $m_i$ be the number chosen by Maker and let $b_i$ be the number chosen by Breaker. Let $m_1=c_1a_1^\ell$ where $c_1$ is power-$\ell$ free. If $\mathcal{B}$ tells Breaker to choose $\alpha_1$ for the $G^*([k^2+2],k,1)$ game given that Maker has selected $a_1$, then Breaker sets $b_1=c_1\alpha_1^\ell$. Consider round $i\geq2$. Suppose Maker has chosen $m_1=c_1a_1^\ell, m_2=c_2a_2^\ell, \ldots, m_i=c_ia_i^\ell$ and Breaker has selected $b_1=c_1\alpha_1^\ell, b_2=c_2\alpha_2^\ell, \ldots,b_{i-1}=c_{i-1}\alpha_{i-1}^\ell$. Let $c_{j_1},c_{j_2},\ldots,c_{j_s}\in\{1,\ldots,i-1\}$ be all the indices such that
\[
c_{j_1}=c_{j_2}=\cdots=c_{j_s}=c_i.
\]
If $\mathcal{B}$ tells Breaker to choose $\alpha_i$ for the $G^*([k^2+2],k,1)$ game given that Maker has has selected $a_{j_1}, a_{j_2}, \ldots, a_{j_s}$, $a_i$ and Breaker has selected $b_{j_1}, b_{j_2}, \ldots, b_{j_s}$, then Breaker sets $b_i=c_i\alpha_i^\ell$.

Since $\mathcal{B}$ is a winning strategy for Breaker, Breaker can stop Maker from completing a solution set from each $A(c)$ and hence wins the game.

\section{Proof of \cref{Theorem:Main-NegativePower-NotDistinct}}\label{Section:Proof-MainTheorem-NegativeNotDistinct}
Let $k,\ell$ be integers with $k\geq2$ and $\ell\leq-1$. We start with an observation. 

\begin{lemma}\label{Lemma:MakerNeedsOneForPower-1}
If $n<2k^{-\ell}$ and Maker does not choose $1$ in the first round, then Breakers wins the $G([n],k,\ell)$ game.
\end{lemma}
\begin{proof}
Suppose $n<2k^{-\ell}$ and Maker does not choose $1$ in the first round. We show that Break wins the $G([n],k,\ell)$ game by choosing $1$ in the first round. Suppose, for a contradiction, that Maker wins. Let $(x_1,\ldots,x_k,y)=(a_1,\ldots,a_k,b)$ be a solution to \cref{Equation:Fractional-Coefficients1} in $\{1,2,\ldots,n\}$ completed by Maker. Then since $a_i\leq n<2k^{-\ell}$ for all $i=1,\ldots,k$, we have
\[
b^{1/\ell}=a_1^{1/\ell}+\cdots+a_k^{1/\ell}>k(2k^{-\ell})^{1/\ell}=2^{1/\ell}.
\]
So $b<2$ which is impossible.
\end{proof}

Now we prove the lower bound in \cref{Theorem:Main-NegativePower-NotDistinct}.
\begin{lemma}
If $k\geq1/(2^{-1/\ell}-1)$, then $f(k,\ell)\geq(k+1)^{-\ell}$.
\end{lemma}
\begin{proof}
Suppose $k\geq1/(2^{-1/\ell}-1)$. It suffices to show that that Breaker wins the $G([(k+1)^{-\ell}-1],k,\ell)$ game. By straightforward calculation, we have
\[
(k+1)^{-\ell}-1< 2k^{-\ell}.
\]
Hence, by \cref{Lemma:MakerNeedsOneForPower-1}, we can assume that Maker chooses $1$ in the first round and $b=1$. Now we show that the only solution to $x_1^{1/\ell}+\cdots+x_k^{1/\ell}=1$ in $\{1,2,\ldots,(k+1)^{-\ell}-1\}$ is $(x_1,\ldots,x_k)=(k^{-\ell},\ldots,k^{-\ell})$. This would imply that Breaker can choose $k^{-\ell}$ in the first round and win the game. Let $a_1,\ldots,a_k\in\{1,2,\ldots,(k+1)^{-\ell}-1\}$ with
\[
a_1^{1/\ell}+\cdots+a_k^{1/\ell}=1,
\]
and $a_1\leq\cdots\leq a_k$. Since the sum a rational number and an irrational number is irrational, $a_1^{1/\ell},\ldots,a_k^{1/\ell}$ are rational numbers. Since $a_1,\ldots,a_k\in\{1,2,\ldots,(k+1)^{-\ell}-1\}$, we have $a_1,\ldots,a_k\in\{1,2^{-\ell},\ldots,k^{-\ell}\}$. If $a_i<k^{-\ell}$ for some $i\in [k]$, then
\[
1=a_1^{1/\ell}+\cdots+a_k^{1/\ell}>k(k^{-\ell})^{1/\ell}=1
\]
which is impossible. Hence the only solution to $x_1^{1/\ell}+\cdots+x_k^{1/\ell}=1$ in $\{1,2,\ldots,(k+1)^{-\ell}-1\}$ is $(x_1,\ldots,x_k)=(k^{-\ell},\ldots,k^{-\ell})$ and Breaker wins the $G([(k+1)^{-\ell}-1],k,\ell)$ game.
\end{proof}

Now we prove the upper bound in \cref{Theorem:Main-NegativePower-NotDistinct}. By \cref{Lemma:Linear-Radical}, $f(k,\ell)\leq[f(k,-1)]^{-\ell}$. Hence, it suffices to show that for all $k\geq4$, $f(k,-1)\leq k+2$. The next two lemmas will establish this.

\begin{lemma}
If $k+1\neq p$ or $p^2$ for any prime $p$, then $f(k,-1)\leq k+1$.
\end{lemma}
\begin{proof}
Suppose $k+1\neq p$ or $p^2$ for any prime $p$. We will prove that Maker wins the $G([k+1],k,-1)$ game. In this case, we have $k+1=rs$ for some integers $r>1$ and $s>1$ with $r\neq s$. Then we have $(r-1)s\neq r(s-1)$, $(r-1)s<k<k+1$, and $r(s-1)<k<k+1$. Consider the following solutions in $\{1,2,\ldots,k+1\}$: 
\[
(x_1,x_2,\ldots,x_{k-1},x_k,y)=(k,k,\ldots,k,k,1),\]
\[
(x_1,\ldots,x_{(r-1)s},x_{(r-1)s+1},\ldots,x_k,y)=(rs,\ldots,rs,r(s-1),\ldots,r(s-1),1),
\]
and
\[
(x_1,\ldots,x_{r(s-1)},x_{r(s-1)+1},\ldots,x_k,y)=(rs,\ldots,rs,(r-1)s,\ldots,(r-1)s,1).
\]
Based on these solutions, Maker wins the $G([k+1],k,-1)$ game using the following strategy: Maker chooses 1 in the first round; if Breaker does not choose $k$ in the first round, then Maker chooses $k$ in the second round to win the game; otherwise, Maker will choose $k+1=rs$ in the second round and win the game by choosing either $r(s-1)$ or $(r-1)s$ in the third round.
\end{proof}

\begin{lemma}\label{Lemma:UpperPower-1Prime}
If $k+1=p$ or $p^2$ for some prime $p\geq5$, then $f(k,-1)\leq k+2$.
\end{lemma}
\begin{proof}
Suppose $k+1=p$ or $p^2$ for some prime $p\geq5$. We show that Maker wins the $G([k+2],k,-1)$ game. 

Since $k+1\geq5$ is odd, $k$ is even and $k\geq4$. Hence $(k+2)/2\neq k$. Consider the following solutions in $\{1,2,\ldots,k+2\}$:
\[
(x_1,x_2,\ldots,x_{k-1},x_k,y)=(k,k,\ldots,k,k,1),
\]
\[
(x_1,\ldots,x_{(k-2)/2},x_{(k-2)/2+1},\ldots,x_k,y)=(k-2,\ldots,k-2,k+2,\ldots,k+2,1),
\]
and
\[
(x_1,x_2,x_3,\ldots,x_k,y)=((k+2)/2,(k+2)/2,k+2,\ldots,k+2,1).
\]
Based on these solutions, Maker wins the $G([k+2],k,-1)$ game by the following strategy: Maker chooses $1$ in the first round; if Breaker does not choose $k$ in the first round, then Maker chooses $k$ in the second round to win the game; otherwise, Maker will choose $k+2$ in the second round and win the game by choosing either $(k+2)/2$ or $k-2$ in the third round.
\end{proof}
\subsection{Remarks}
The inequality in \cref{Lemma:UpperPower-1Prime} becomes equality when $k+1=p$ for some odd prime $p$.

\begin{theorem}
If $k+1=p$ for some odd prime $p$, then $f(k,-1)=k+2$.
\end{theorem}
\begin{proof}
Suppose $k+1=p$ for some odd prime. By \cref{Lemma:UpperPower-1Prime}, we have $f(k,-1)\leq k+2$. It remains to show that $f(k,-1)\geq k+2$. To do this, it suffices to show that Breaker wins the $G([k+1],k,-1)$ game.

Case 1: $k+1=3$. The only solution to $1/x_1+\cdots+1/x_k=1/y$ in $\{1,2,3\}$ with $x_1,\ldots, x_k$ not necessarily distinct is $(x_1,x_2,y)=(2,2,1)$. Hence Breaker can win by choosing either $1$ or $2$ in the first round.

Case 2: $k+1\geq5$. By \cref{Lemma:MakerNeedsOneForPower-1}, if Maker does not choose $1$ in the first round, then Breaker wins. So we assume that Maker chooses $1$ in the first round. Now we show that Breaker wins by choosing $k$ in the first round. It suffices to show that $\{1,2,\ldots,k-1,k+1\}$ does not have a solution to $1/x_1+\cdots+1/x_k=1/1$ where $x_1$, $\ldots$, $x_k$ are not necessarily distinct. Suppose $(x_1,x_2,\ldots,x_{k-1},x_k)=(a_1,a_2,\ldots,a_{k-1},a_k)$ is a solution in $\{1,2,\ldots,k-1,k+1\}$. We show that $a_k=k+1$. Suppose not. Then $a_i<k$ for all $i=1,2,\ldots,k$. So
\[
\frac{1}{a_1}+\cdots+\frac{1}{a_k}>\frac{1}{k}+\cdots+\frac{1}{k}=\frac{1}{1}
\]
which is a contradiction. Hence $a_k=k+1$. Now we have
\[
1=\frac{r}{k+1}+\sum_{i=1}^{k-r}\frac{1}{a_i}
\]
where $r\in\{1,2,\ldots,k-1\}$ and $a_i<k$ for all $i=1,\ldots,k-r$. Rearranging the equation, we get
\[
\sum_{i=1}^{k-r}\frac{1}{a_i}=\frac{p-r}{p}.
\]
Since $p$ is prime, $p$ divides the least common multiple of $a_1,\ldots,a_{k-r}$. Since $p$ is prime, $p$ divides $a_i$ for some $i$ which is a contradiction because $a_i<p$ for all $i$. Hence Breaker wins the game.
\end{proof}
We are unable to verify that $f(k,-1)=k+2$ when $k+1=p^2$ for some odd prime $p$. However, we believe this should be the case.
\begin{conjecture}
If $k+1=p^2$ for some odd prime $p$, then $f(k,-1)=k+2$.
\end{conjecture}

\section{Proof of \cref{Theorem:Main-NegativePower-Distinct}}\label{Section:Proof-MainTheorem-NegativeDistinct}
Let $k,\ell$ be integers with $k\geq2$ and $\ell\leq-1$. By \cref{Lemma:Linear-Radical}, we have $f^*(k,\ell)\leq[f^*(k,-1)]^{-\ell}$. It remains to show that $f^*(k,-1)=\exp(O_k(k\log k))$. By \cref{Theorem:GameBrownRodl}, it suffices to find a finite set $A\subseteq\mathbb{N}$ such that Maker wins the $G^*(A,x_1+\cdots+x_k=y)$ game and the least common multiple of $A$ is small.

\begin{lemma}\label{Lemma:Power-1Distinct}
Let $k\geq4$ be an integer and let $A=\{1,\ldots,2k+1\}\cup\{k^2-k+1,\ldots,k^2+2k\}$. Then Maker wins the $G^*(A,x_1+\cdots+x_k=y)$ game.
\end{lemma}
\begin{proof}
Let $k\geq4$. For $i=1,\ldots,k+3$, let $m_i$ be the number selected by Maker in round $i$ and let $b_i$ be the number selected by Breaker in round $i$.

Consider the following strategy for Maker:
\begin{itemize}
\item[(1)] Set $m_1=1$ and $M_1=\{\{2,3\},\{4,5\},\ldots,\{2k,2k+1\}\}$.
\item[(2)] For $i=2,\ldots,k+1$, if $b_{i-1}\in B$ for some $B\in M_{i-1}$, then set $m_i\in B\backslash\{b_{i-1}\}$ and $M_i=M_{i-1}\backslash\{B\}$; if $b_{i-1}\notin B$ for any $B\in M_{i-1}$, then set $m_i=\min_{S\in M_{i-1}}\min S$, $M_i=M_{i-1}\backslash S'$ where $m_i\in S'$.
\item[(3)] In round $k+2$, if there exists a subset $\{a_1,\ldots,a_k\}\subseteq\{m_1,\ldots,m_{k+1}\}$ of size $k$ such that $a_1+\cdots+a_k\in\{k^2-k+1,\ldots,k^2+2k\}\backslash\{b_1,\ldots,b_{k+1}\}$, then set $m_{k+2}=a_1+\cdots+a_k$. Otherwise, set $m_{k+2}=2k+1$, and then, in round $k+3$, set $m_{k+3}=a_1+\cdots+a_k$ where $\{a_1,\ldots,a_k\}\subseteq\{m_1,\ldots,m_{k+2}\}$ has size $k$ with $a_1+\cdots+a_k\in\{k^2-k+1,\ldots,k^2+2k\}\backslash\{b_1,\ldots,b_{k+2}\}$.
\end{itemize}
In Step (3), Maker wins for the first case. So it remains to show that if no subset $\{a_1,\ldots,a_k\}\subseteq\{m_1,\ldots,m_{k+1}\}$ of size $k$ satisfies $a_1+\cdots+a_k\in\{k^2-k+1,\ldots,k^2+2k\}\backslash\{b_1,\ldots,b_{k+1}\}$, then Maker can set $m_{k+2}=2k+1$ in round $k+2$ and there exists a subset $\{a_1,...,a_k\}\subseteq\{m_1,\ldots,m_{k+2}\}$ of size $k$ such that $a_1+\cdots+a_k\in\{k^2-k+1,\ldots,k^2+2k\}\backslash\{b_1,\ldots,b_{k+2}\}$.

Suppose, at the beginning of round $k+2$, no subset $\{a_1,\ldots,a_k\}\subseteq\{m_1,\ldots,m_{k+1}\}$ of size $k$ satisfies $a_1+\cdots+a_k\in\{k^2-k+1,\ldots,k^2+2k\}\backslash\{b_1,\ldots,b_{k+1}\}$. First note that by Maker's strategy, for all $i=2,\ldots,k+1$, $m_i=2(i-1)$ or $2(i-1)+1$. So for all subsets $\{a_1,\ldots,a_k\}\subseteq\{m_1,\ldots,m_{k+1}\}$ of size $k$, we have
\[
a_1+\cdots+a_k\geq1+2+4+\cdots+2(k-1)=k^2-k+1
\]
and
\[
a_1+\cdots+a_k\leq 3+5+\cdots+2k+1=(k+1)^2-1=k^2+2k.
\]
So if no subset $\{a_1,\ldots,a_k\}\subseteq\{m_1,\ldots,m_{k+1}\}$ of size $k$ satisfies $a_1+\cdots+a_k\in\{k^2-k+1,\ldots,k^2+2k\}\backslash\{b_1,\ldots,b_{k+1}\}$, then $b_1,\ldots,b_{k+1}\notin\{1,\ldots,2k+1\}$. Now according to Maker's strategy, we have, $m_1=1$, and $m_i=2(i-1)$ for all $i=2,\ldots,k+1$. This implies that at the beginning of round $k+2$, $2k+1$ is available to Maker and hence Maker can set $m_{k+2}=2k+1$. At the same time, for all subsets $\{a_1,\ldots,a_k\}\subseteq\{m_1,\ldots,m_{k+1}\}$ of size $k$, we have $a_1+\cdots+a_k\leq 2+4+\cdots+2k=k^2+k$ and hence $b_1,\ldots,b_{k+1}\leq k^2+k$. By setting $m_{k+2}=2k+1$, there are at least two subsets of $\{m_1,\ldots,m_{k+2}\}$ of size $k$ whose sum is greater than $k^2+k$. They are $\{2,4,\ldots,2(k-1),2k+1\}$ and $\{2,4,\ldots,2(k-2),2k,2k+1\}$. The first subset sums to $k^2+k+1<k^2+2k$ and the second one sums to $k^2+k+3<k^2+2k$. Since Breaker can only occupy one of them in round $k+2$, there exists a subset $\{a_1,\ldots,a_k\}\subseteq\{m_1,...,m_{k+2}\}$ of size $k$ such that $a_1+\cdots+a_k\in\{k^2-k+1,\ldots,k^2+2k\}\backslash\{b_1,\ldots,b_{k+2}\}$. This proves that Maker wins the $G^*(A,x_1+\cdots+x_k=y)$ game.
\end{proof}

Let $k\geq4$ be an integer and let $A:=\{1,\ldots,2k+1\}\cup\{k^2-k+1,\ldots,k^2+2k\}$. By \cref{Theorem:GameBrownRodl,Lemma:Power-1Distinct}, we have
\[
\begin{split}
f^*(k,-1)\leq&\text{lcm}\{n:n\in A\}\\\leq&\text{lcm}\{1,...,2k+1\}\text{lcm}\{k^2-k+1,...,k^2+2k\}\\\leq&\text{lcm}\{1,...,2k+1\}(k^2+2k)^{3k}\\=&e^{(2+o_k(1))k}e^{3k\log(k^2+2k)}.
\end{split}
\]
Hence we have $f^*(k,-1)=\exp(O_k(k\log k))$.

\subsection{Remarks} By exhaustive search, we are able to find the exact value of $f^*(k,-1)$ for $k=2$.

\begin{proposition}
$f^*(2,-1)=36.$
\end{proposition}
\begin{proof}
We first show that Maker wins the $G^*([36],2,-1)$ game. Consider the following solutions to $1/x_1+1/x_2=1/y$ in $\{1,2,\ldots,36\}$ with $x_1\neq x_2$: $(x_1,x_2,y)=(4,12,3)$, $(6,12,4)$, $(12,36,9)$, and $(18,36,12)$. We construct a rooted binary tree using these solutions as follows:

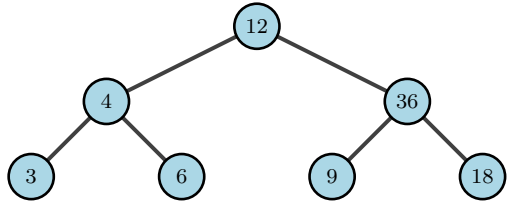
\begin{figure}[H]
\centering
\begin{tikzpicture}
\Vertex[x=0,y=0,label=$3$]{1};
\Vertex[x=2,y=0,label=$6$]{2};
\Vertex[x=4,y=0,label=$9$]{3};
\Vertex[x=6,y=0,label=$18$]{4};
\Vertex[x=1,y=1,label=$4$]{5};
\Vertex[x=5,y=1,label=$36$]{6};
\Vertex[x=3,y=2,label=$12$]{7}
\Edge(7)(5);
\Edge(7)(6);
\Edge(5)(1);
\Edge(5)(2);
\Edge(6)(3);
\Edge(6)(4);
\end{tikzpicture}
\caption{Rooted Binary Tree for Solutions to $1/x_1+1/x_2=1/y$}\label{Figure:Tree}
\end{figure}

In \cref{Figure:Tree}, each path from the root 12 to a leaf is a solution set to $1/x_1+1/x_2=1/y$. It is easy to see that Maker can win this game by doing the following:
\begin{itemize}
\item[(1)] Maker selects the root in round 1.
\item[(2)] In round 2, Maker selects a vertex that is adjacent to the root such that both of its children are untouched by Breaker.
\item[(3)] In round 3, Maker chooses a child of the vertex that Maker selected in round 2.
\end{itemize}

Now we show that Breaker wins the $G^*([35],2,-1)$ game. One can check that there are 13 solutions to $1/x_1+1/x_2=1/y$ in $\{1,2,\ldots,35\}$: $\{2,3,6\}$, $\{3,4,12\}$, $\{4,6,12\}$, $\{4,5,20\}$, $\{5,6,30\}$, $\{6,8,24\}$, $\{6,9,18\}$, $\{6,10,15\}$, $\{8,12,24\}$, $\{10,14,35\}$, $\{10,15,30\}$, $\{12,20,30\}$, and $\{12,21,28\}$. Breaker wins the game using the pairing strategy over $\{4,12\}$, $\{8,24\}$, $\{10,15\}$, $\{2,3\}$, $\{5,20\}$, $\{6,30\}$, $\{9,18\}$, $\{14,35\}$, $\{20,30\}$, and $\{21,28\}$.
\end{proof}

For general $k$, \cref{Theorem:Main-NegativePower-Distinct} only provides an upper bound for $f^*(k,-1)$. It is trivially true that $f^*(k,-1)\geq2k+1$ because Maker needs to occupy at least $k+1$ numbers to win. However, we don't have a nontrivial lower bound.

\begin{oproblem}
Find a nontrivial lower bound for $f^*\left(k,-1\right)$.
\end{oproblem}
\section{Equations with Arbitrary Coefficients}\label{Section:AribitraryCoefficients}
In this section, we briefly discuss the Maker-Breaker Rado games for the equation
\begin{equation}\label{Equation:AribtraryCoefficients}
a_1x_1+\cdots+a_kx_k=y,
\end{equation}
where $k,a_1,\ldots,a_k$ are positive integers with $k\geq2$ and $a_1\geq a_2\geq\cdots\geq a_k$. Write $w:=a_1+\cdots+a_k$, and $w^*:=\sum_{i=1}^k(2i-1)a_i$. Let $f(a_1,\ldots,a_k;y)$ be the smallest positive integer $n$ such that Maker wins the $G([n],a_1x_1+\cdots+a_kx_k=y)$ game and let $f^*(a_1,\ldots,a_k;y)$ be the smallest positive integer $n$ such that Maker wins the $G^*([n],a_1x_1+\cdots+a_kx_k=y)$ game. 

Hopkins and Schaal \cite{HopkinsSchaal2005}, and Guo and Sun \cite{GuoSun2008} proved that if $\{1,2,\ldots,a_kw^2+w-a_k\}$ is partitioned into two classes, then one of them contains a solution to \cref{Equation:AribtraryCoefficients} with $x_1,\ldots,x_k$ not necessarily distinct; and there exists a partition of $\{1,2,\ldots,a_kw^2+w-a_k-1\}$ into two classes such that neither contains a solution to \cref{Equation:AribtraryCoefficients} with $x_1,\ldots,x_k$ not necessarily distinct. By these results and strategy stealing, we have $f(a_1,\ldots,a_k;y)\leq a_kw^2+w-a_k$. The strategy stealing argument here is similar to the one in \cref{Section:Introduction} where we explained that $f(k,\ell)\leq R(k,\ell)$ and $f^*(k,\ell)\leq R^*(k,\ell)$. The next theorem shows that, in fact, $f(a_1,\ldots,a_k;y)$ is much smaller than $a_kw^2+w-a_k$.

\begin{theorem}\label{Theorem:AritraryCoefficients-NonDistinct}
For all integers $k\geq2$, we have $w+2a_k\leq f(a_1,\ldots,a_k;y)\leq w+a_{k-1}+a_k$.
\end{theorem}
\begin{proof}
The case that $k=2$ and $a_1=a_2=1$ is a special case of \cref{Lemma:Power1-NonDistinct}. So we assume that $k>2$ or $k=2$ but $a_1\geq2$. Then $w>2$.

We first show that Maker wins the $G([w+a_{k-1}+a_k],a_1x_1+\cdots+a_kx_k=y)$ game. Maker chooses $1$ in round 1. If Breaker does not choose $w$ in round 1, then Maker wins in round 2 by choosing $w$. If Breaker chooses $w$ in round 1, then Maker chooses 2 in round 2 and either $w+a_k$ or $w+a_{k-1}+a_k$ in round 3 to win the game.

Now we show that Breaker wins the $G([w+2a_k-1],a_1x_1+\cdots+a_kx_k=y)$ game. The only solutions to \cref{Equation:AribtraryCoefficients} in $\{1,2,\ldots,w+2a_k-1\}$ are 
\[
(x_1,x_2,\ldots,x_{k-1},x_k,y)=(1,1,\ldots,1,1,w)\] and \[(x_1,x_2,\ldots,x_{k-1},x_k,y)=(1,1,\ldots,1,2,w+a_k).\]
Now Breaker wins by the pairing strategy over $\{1,w\}$ and $\{2,w+a_k\}$. Note that if $a_i=a_k$ for some $i\in\{1,2,\ldots,k-1\}$, then $(x_1,\ldots,x_{i-1},x_i,x_{i+1},\ldots,x_k,y)=(1,\ldots,1,2,1,\ldots,1,w+a_1)$ is also a solution, but Breaker can still win the game by the pairing strategy becuase $w+a_i=w+a_k$.
\end{proof}
The next theorem provides lower and upper bounds for $f^*(a_1,\ldots,a_k;y)$.

\begin{theorem}\label{Theorem:AribitraryCoeffients-Distinct}
For all integers $k\geq4$, we have
\[
w^*\leq f^*(a_1,\ldots,a_k;y)\leq w^*+(2k-2)(a_1-a_k)+(k+3)a_{k-2}.
\]
\end{theorem}
\begin{proof}
Let $k\geq4$ be an integer and write $W=w^*+(2k-2)(a_1-a_k)+(k+3)a_{k-2}$. We first show that Breaker wins the $G^*([w^*-1],a_1x_1+\cdots+a_kx_k=y)$ game by choosing the smallest number available each round. Suppose, for a contradiction, that Maker wins. Let $\alpha_1\leq\alpha_2\leq\cdots\leq\alpha_s$, where $s\geq k+1$, be the numbers chosen by Maker after winning the game. Then by Breaker's strategy, we have $\alpha_i\geq 2i-1$ for all $i=1,2,\ldots,k$. By the rearrangement inequality \cite{HLP1952}, the smallest $k$-sum is
\[
\sum_{i=1}^k a_i\alpha_i\geq\sum_{i=1}^k(2i-1)a_i=w^*
\]
which is a contradiction.

Now we show that Maker wins the $G^*([W],a_1x_1+\cdots+a_kx_k=y)$ game. We split it into two cases.

Case 1: $\alpha_1=\alpha_k=c$ for some $c$. Since the coefficients of $x_1,\ldots,x_k$ are the same, Maker's strategy defined in the proof of \cref{Lemma:Power1Distinct-Upper} still applies by multiplying the $k$-sums in the proof of \cref{Lemma:Power1Distinct-Upper} by $c$. So Maker wins the $G^*([ck^2+3c],a_1x_1+\cdots+a_kx_k=y)$ game. Since
\[
W=w^*+(2k-2)(a_1-a_k)+(k+3)a_{k-2}=ck^2+ck+3c>ck^2+3c,
\]
Maker wins the $G^*([W],a_1x_1+\cdots+a_kx_k=y)$ game.

Case 2: $a_1>a_k$. We will show that Maker wins the game with the following strategy: 
\begin{itemize}
    \item[(1)] Maker chooses the smallest number available each round for the first $k+1$ rounds;
    \item[(2)] and then chooses an available $k$-sum in round $k+2$.
\end{itemize}

For $i=1,2,\ldots,k+1$, let $m_i$ be the number chosen by Maker in round $i$. Then by Maker's strategy, we have $i\leq m_i\leq2i-1$ for all $i=1,2,\ldots,k+1$.

Since $a_1>a_k$, there exists $t\in\{2,3,\ldots,k\}$ such that $\alpha_t<\alpha_{t-1}$. For $i=1,\ldots,k+1$, let $m_i$ be the number chosen by Maker in round $i$. By the rearrangement inequality, we have the following $k$ distinct $k$-sums involving only $m_1,\ldots,m_k$:
\[
(a_tm_{t+j}+a_{t+j}m_t)-(a_tm_t+a_{t+j}m_{t+j})+\sum_{i=1}^ka_im_i,\text{ where } j=0,1,\ldots,k-t
\]
and
\[
(a_{t-j'}m_k+a_km_{t-j'})-(a_{t-j'}m_{t-j'}+a_km_k)+\sum_{i=1}^ka_im_i, \text{ where }j'=1,2,\ldots,t-1.
\]

Among these distinct $k$-sums, the smallest is $\sum_{i=1}^ka_im_i$ and the largest is
\begin{equation}\label{Equaion:Long}
(a_1m_k+a_km_1)-(a_1m_1+a_km_k)+\sum_{i=1}^ka_im_i=a_1m_k+\left(\sum_{i=2}^{k-1}a_im_i\right)+a_km_1.
\end{equation}

Since $k\geq4$, there are two terms of the form $a_im_i$, $i\in\{2,\ldots,k-1\}$, in the middle of the right hand side of \cref{Equaion:Long}. Replacing $m_{k-1}$ with $m_{k+1}$ and replacing $m_{k-2}$ with $m_{k+1}$, we get two larger and distinct $k$-sums:
\[
a_1m_k+\left(\sum_{i=2}^{k-2}a_im_i\right)+a_{k-1}m_{k+1}+a_km_1
\]
and
\[
a_1m_k+\left(\sum_{i=2}^{k-3}a_im_i\right)+a_{k-2}m_{k+1}+a_{k-1}m_{k-1}+a_km_1.
\]
The largest of these $k$-sums is
\[
\begin{split}
&a_1m_k+\left(\sum_{i=2}^{k-3}a_im_i\right)+a_{k-2}m_{k+1}+a_{k-1}m_{k-1}+a_km_1\\=&a_1m_k++a_{k-2}m_{k+1}+a_km_1-a_1m_1-a_{k-2}m_{k-2}-a_km_k+\sum_{i=1}^ka_im_i\\=&(m_k-m_1)(a_1-a_k)+a_{k-2}(m_{k+1}-m_{k-2})+\sum_{i=1}^ka_im_i\\\leq&w^*+[(2k-1)-1](a_1-a_k)+[2k+1-(k-2)]a_{k-2}\\=&w^*+(2k-2)(a_1-a_k)+(k+3)a_{k-2}=W.
\end{split}
\]
So there exists a $k$-sum unoccupied by Breaker in the beginning of round $k+2$ and hence Maker wins the $G^*([W],a_1x_1+\cdots+a_kx_k=y)$ game by choosing the available $k$-sum in round $k+2$.
\end{proof}
The bounds in \cref{Theorem:AribitraryCoeffients-Distinct} can be optimized using the technique in the proofs of \cref{Lemma:Power1Distinct-Upper,Lemma:Power1Distinct-Lower}, but we don't attempt it here.

\section{Concluding Remarks}\label{Section:Conclusion}
It would be interesting to study Rado games for other well-studied equations in arithmetic Ramsey theory. One direction is to study Rado games for
\begin{equation}
a_1x_1^{1/\ell}+\cdots+a_kx_k^{1/\ell}=y^{1/\ell},
\end{equation}
where $\ell,k,a_1,\ldots,a_k$ are positive integers with $k\geq2$ and $\ell\neq0$. We studied the $G([n],a_1x_1+\cdots+a_kx_k=y)$ and $G^*([n],a_1x_1+\cdots+a_kx_k=y)$ games in \cref{Section:AribitraryCoefficients}, but how the fractional power $1/\ell$ interacts with the coefficients $a_1,\ldots,a_k$ is yet unknown.
\begin{oproblem}
What is the smallest integer $n$ such that Maker wins the $G([n],a_1x_1^{1/\ell}+\cdots+a_kx_k^{1/\ell}=y^{1/\ell})$ game for $\ell\in\mathbb{Z}\backslash\{-1,0,1\}$? And what is the smallest integer $n$ such that Maker wins the $G^*([n],a_1x_1^{1/\ell}+\cdots+a_kx_k^{1/\ell}=y^{1/\ell})$ game for $\ell\in\mathbb{Z}\backslash\{-1,0,1\}$?
\end{oproblem}

Another direction is to study Rado games for the equation
\begin{equation}\label{Equation:PowerMoreThan1}
x_1^\ell+\cdots+x_k^\ell=y^\ell
\end{equation}
where $\ell\in\mathbb{Z}\backslash\{-1,0,1\}$ and $k\in\mathbb{N}\backslash\{1\}$. In 2016, Heule, Kullmann, and Marek \cite{HKM2016} verified that if $\{1,2,\ldots,7825\}$ is partitioned into two classes, then one of them contains a solution to \cref{Equation:PowerMoreThan1} with $k=\ell=2$ and that there exists a  partition of $\{1,2,\ldots,7824\}$ into two classes so that neither contains a solution to \cref{Equation:PowerMoreThan1} with $k=\ell=2$. It is easy to see that if $a_1,a_2,b\in\mathbb{N}$ with $a_1^2+a_2^2=b^2$, then $a_1\neq a_2$. So the result of Heule, Kullmann, and Marek implies that Maker wins both the $G([7825],x_1^2+x_2^2=y^2)$ game and the $G^*([7825],x_1^2+x_2^2=y^2)$ game. It would be interesting to see if Maker can do better.

\begin{oproblem}
Does there exist $n<7825$ such that Maker wins the $G^*([n],x_1^2+x_2^2=y^2)$ game?
\end{oproblem}

The situation for Maker is more complicated when $\ell\geq3$. By Fermat's last theorem \cite{Wiles1995}, for all $n,\ell\in\mathbb{N}$ with $\ell\geq3$, Breaker wins both the $G([n],x_1^\ell+x_2^\ell=y^\ell)$ game and the $G^*([n],x_1^\ell+x_2^\ell=y^\ell)$ for $\ell\geq3$. By homogeneity, Breaker also wins the $G([n],x_1^\ell+x_2^\ell=y^\ell)$ game and the $G^*([n],x_1^\ell+x_2^\ell=y^\ell)$ game for all $n\in\mathbb{N}$ and $\ell\leq-3$. Hence, in order to study Rado games for \cref{Equation:PowerMoreThan1}, one needs extra conditions on $k$ and $\ell$ to make sure there are solutions to \cref{Equation:PowerMoreThan1} in $\mathbb{N}$. Recently, Chow, Lindqvist, and Prendiville \cite{CLP2021} proved that, for all $\ell\in\mathbb{N}$, there exists $k_0\in\mathbb{N}$ such that for all $k\geq k_0$, if we partition of $\mathbb{N}$ into two classes, then one of them contains a solution to \cref{Equation:PowerMoreThan1} with $x_1,\ldots,x_k$ not necessarily distinct. By the result of Brown and R\"{o}dl \cite{BrownRodl1991} described in \cref{Section:Introduction}, the same result holds for $\ell\in\{-1,-2,\ldots\}$ as well. For example, if $|\ell|=2$, then $k=4$ suffices; and if $|\ell|=3$, then $k=7$ is enough.

\section*{Acknowledgements}
The authors would like to thank the referee for their suggestions which helped improve the presentation of the paper.

\end{document}